\documentclass[a4paper,reqno,11pt]{amsart}
\usepackage{amssymb}
\usepackage{pifont} 
\usepackage{bbding}
\usepackage{wasysym}
\usepackage{ifthen}
\usepackage{cite}
 \usepackage[dvips]{graphicx}

\nonstopmode \numberwithin{equation}{section}
\usepackage{amssymb}
\usepackage{ifthen}
\usepackage{graphicx}
\usepackage{amsmath}
\usepackage[T1]{fontenc} 
\usepackage[utf8]{inputenc}
\usepackage[usenames,dvipsnames]{color}
\usepackage{color}
\usepackage[english]{babel}
\usepackage{fancyhdr}
\usepackage{fancybox}
\usepackage{tikz}

\nonstopmode \numberwithin{equation}{section}
\theoremstyle{plain}

\newtheorem{conj}{Conjecture}

\theoremstyle{definition}

\newtheorem{thm}{Theorem}[section]
\newtheorem{prob}{Problem}[section]
\newtheorem{cor}{Corollary}[section]

\newtheorem{prop}{Proposition}[section]
\newtheorem{rem}{Remark}[section]
\newtheorem{lem}{Lemma}[section]

\newcounter{minutes}
\divide\time by 60
\newcounter{hours}
\multiply\time by 60
\addtocounter{minutes}{-\time}

\newcounter {own}
\def\theown {\thesection       .\arabic{own}}

\newenvironment{pf}[1][]{%
 \vskip 3mm
 \noindent
 \ifthenelse{\equal{#1}{}}%
  {{\slshape Proof. }}%
  {{\slshape #1.} }%
 }%
{\qed\bigskip}

\newcounter{alphabet}

\def\be{\begin{equation}}
\def\ee{\end{equation}}

\newcommand{\bee}{\begin{enumerate}}
\newcommand{\eee}{\end{enumerate}}

\newcommand{\blem}{\begin{lem}}
\newcommand{\elem}{\end{lem}}
\newcommand{\bthm}{\begin{thm}}
\newcommand{\ethm}{\end{thm}}
\newcommand{\bcor}{\begin{cor}}
\newcommand{\ecor}{\end{cor}}
\newcommand{\beg}{\begin{examp}}
\newcommand{\eeg}{\end{examp}}
\newcommand{\begs}{\begin{examples}}
\newcommand{\eegs}{\end{examples}}

\newcommand{\bdefn}{\begin{defn}}
\newcommand{\edefn}{\end{defn}}

\newcommand{\bprob}{\begin{prob}}
\newcommand{\eprob}{\end{prob}}
\newcommand{\bei}{\begin{itemize}}
\newcommand{\eei}{\end{itemize}}

\newcommand{\bcon}{\begin{conj}}
\newcommand{\econ}{\end{conj}}
\newcommand{\bcons}{\begin{conjs}}
\newcommand{\econs}{\end{conjs}}
\newcommand{\bprop}{\begin{prop}}
\newcommand{\eprop}{\end{prop}}
\newcommand{\br}{\begin{rem}}
\newcommand{\er}{\end{rem}}
\newcommand{\brs}{\begin{rems}}
\newcommand{\ers}{\end{rems}}
\newcommand{\bo}{\begin{obser}}
\newcommand{\eo}{\end{obser}}
\newcommand{\bos}{\begin{obsers}}
\newcommand{\eos}{\end{obsers}}
\newcommand{\bpf}{\begin{pf}}
\newcommand{\epf}{\end{pf}}
\newcommand{\ba}{\begin{array}}
\newcommand{\ea}{\end{array}}
\newcommand{\beq}{\begin{eqnarray}}
\newcommand{\beqq}{\begin{eqnarray*}}
\newcommand{\eeq}{\end{eqnarray}}
\newcommand{\eeqq}{\end{eqnarray*}}

\begin{document}

\title{Sharp Coefficient and Inverse Problems for Holomorphic Semigroup Generators}

\author{Sanju Mandal}
\address{Sanju Mandal, Department of Mathematics, Jadavpur University, Kolkata-700032, West Bengal,India.}
\email{sanju.math.rs@gmail.com, sanjum.math.rs@jadavpuruniversity.in}

\author{Molla Basir Ahamed$^*$}
\address{Molla Basir Ahamed, Department of Mathematics, Jadavpur University, Kolkata-700032, West Bengal,India.}
\email{mbahamed.math@jadavpuruniversity.in}

\thanks{*~Corresponding author email: mbahamed.math@jadavpuruniversity.in}

\subjclass[{AMS} Subject Classification:]{Primary 30C45; 30C50 Secondary 37F10}
\keywords{Holomorphic generators, Semigroup generators, Univalent functions, Logarithmic coefficients, successive coefficient difference, Fekete–Szeg\"o functional}

\def\thefootnote{}
\footnotetext{ {\tiny File:~\jobname.tex,
printed: \number\year-\number\month-\number\day,
          \thehours.\ifnum\theminutes<10{0}\fi\theminutes }
} \makeatletter\def\thefootnote{\@arabic\c@footnote}\makeatother

\begin{abstract}
	In this paper, we study extremal problems for coefficient functionals associated with a distinguished subclass of holomorphic semigroup generators, denoted by $\mathcal{A}_{\beta}$ ($0 \le \beta \le 1$), defined on the unit disk $\mathbb{D}$. This class forms a natural filtration of the class $\mathcal{G}_0$ of infinitesimal generators, with the class $\mathcal{R}$ of functions of bounded turning arising as its minimal element.	We obtain sharp bounds for the initial logarithmic coefficients $\gamma_n$, the inverse coefficients $A_n$, and the logarithmic inverse coefficients $\Gamma_n$ for $n = 1,2,3$ within the class $\mathcal{A}_{\beta}$. In addition, we address the successive coefficient problem by deriving sharp upper and lower estimates for the differences $|A_{n+1}| - |A_n|$ for $n = 1,2$. Furthermore, we establish sharp bounds for a generalized Fekete--Szeg\"o functional in the class $\mathcal{R}$. The extremality of the obtained results is demonstrated by explicit constructions, including functions related to Gauss hypergeometric functions. Our results unify and extend several earlier contributions in geometric function theory and reveal a structural connection between coefficient problems for functions of bounded turning and the dynamics of holomorphic semigroup generators.
\end{abstract}

\maketitle
\pagestyle{myheadings}
\markboth{S. Mandal and M. B. Ahamed}{Coefficient and Inverse Estimates for Semigroup Generators}

\section{\bf Introduction}
Let $\mathcal{H}$ denote the class of functions that are holomorphic in the open unit disk $\mathbb{D} = \{ z \in \mathbb{C} : |z| < 1 \}$. Equipped with the topology of uniform convergence on compact subsets of $\mathbb{D}$, the space $\mathcal{H}$ becomes a locally convex topological vector space.

We consider the standard subclass
\[
\mathcal{A} = \{ f \in \mathcal{H} : f(0) = 0 \ \text{and} \ f'(0) = 1 \}.
\]
Each function $f \in \mathcal{A}$ admits the Taylor expansion
\begin{equation}\label{eq-1.1}
	f(z) = z + \sum_{n=2}^{\infty} a_n z^n, \quad z \in \mathbb{D}.
\end{equation}
Let $\mathcal{S} \subset \mathcal{A}$ denote the class of univalent functions in $\mathbb{D}$. For background on the theory of univalent functions, we refer to the classical monographs \cite{Duren-1983-NY, Goodman-1983}.\vspace{2mm}

Let $\mathcal{B}$ be the class of holomorphic self mappings from $\mathbb{D}$ to $\mathbb{D}$. A family $\{u_t(z)\}_{t\geq 0}\subset\mathcal{B}$ is called a one-parameter continuous semigroup if 
\begin{enumerate}
	\item[(i)] $\lim\limits_{t\to 0} u_t(z)=z$.
	\item[(ii)] $u_{t+s}(z)=u_t(z)u_s(z)$,
	\item[(iii)]  $\lim\limits_{t\to s}u_t(z)=u_s(z)$
\end{enumerate}
for each $z\in\mathbb{D}$ hold.\vspace{1.2mm}

In their seminal work, Berkson and Porta \cite{Berkson-Porta-MMJ-1978} showed that every one-parameter semigroup $\{u_t\}_{t \ge 0}$ of holomorphic self-maps of $\mathbb{D}$ is locally differentiable with respect to the parameter $t \ge 0$. Moreover, if
\[
\lim_{t \to 0} \frac{z - u_t(z)}{t} = f(z),
\]
where $f$ is holomorphic in $\mathbb{D}$, then $\{u_t\}$ is the unique solution to the Cauchy problem
\[
\frac{\partial u_t(z)}{\partial t} + f(u_t(z)) = 0, \qquad u_0(z) = z.
\]
The function $f$ is called the (holomorphic) generator of the semigroup $\{u_t\}_{t \ge 0} \subset \mathcal{B}$, and the class of all such generators is denoted by $\mathcal{G}$.

It is worth noting that, although each mapping $u_t$ is univalent in $\mathbb{D}$, the corresponding generator $f$ need not be univalent (see \cite{Elin-Shoiket-2010}). The structural and dynamical properties of such generators and their associated semigroups have been extensively studied; see, for instance, \cite{Berkson-Porta-MMJ-1978, Bracci-Contreras-Diaz-2020, Elin-Shoiket-2010, Elin-Shoikhet-Sugawa-2018, Elin-Reich-Shoikhet-2019, Shoikhet-2001} and the references therein.

A fundamental characterization of these generators, due to Berkson and Porta \cite{Berkson-Porta-MMJ-1978}, is recalled below.
\begin{thm}\cite{Berkson-Porta-MMJ-1978}
	The following assertions are equivalent:
	\begin{enumerate}
		\item[(a)] $f\in\mathcal{G}$,
		\item[(b)] $f(z)=(z-\sigma)(1-z\bar{\sigma})p(z)$ with some $\sigma\in\overline{\mathbb{D}}$ and $p\in\mathcal{H}$,\; ${\rm Re}\; p(z)\geq 0$.  
	\end{enumerate}
\end{thm}
The point $\sigma \in \overline{\mathbb{D}} := \{ z \in \mathbb{C} : |z| \le 1 \}$ is called the Denjoy--Wolff point of the semigroup generated by $f$. By the Denjoy--Wolff theorem for continuous semigroups (see \cite{Shoikhet-2001}), if the semigroup $\{u_t\}_{t \ge 0}$ generated by $f$ contains at least one element that is neither an elliptic automorphism of $\mathbb{D}$ nor the identity mapping, then there exists a unique point $\sigma \in \overline{\mathbb{D}}$ such that
\[
\lim_{t \to \infty} u_t(z) = \sigma
\]
uniformly on compact subsets of $\mathbb{D}$, for each $z \in \mathbb{D}$.

We denote by $\mathcal{G}[\sigma]$ the class of holomorphic generators with Denjoy--Wolff point $\sigma$. In the particular case $\sigma = 0$, this reduces to the subclass
\[
\mathcal{G}[0] = \{ f \in \mathcal{G} : f(z) = z p(z), \ \Re p(z) \ge 0 \}.
\]
Following \cite{Bracci et. al.-FACM-2018}, we define $\mathcal{G}_0 := \mathcal{G}[0] \cap \mathcal{A}$. This class plays a central role in the study of non-autonomous problems, particularly in Loewner theory; see, for instance, \cite{Bracci-Contreras-Diaz-JRAM-2012, Duren-1983-NY}.

Recently, several parametrized subclasses of $\mathcal{G}_0$, forming a filtration with the class
\[
\mathcal{R} = \{ f \in \mathcal{A} : \Re f'(z) > 0 \}
\]
of functions of bounded turning as its minimal element, have been investigated (see \cite{Bracci et. al.-FACM-2018, Elin-Shoikhet-Sugawa-2018, Shoikhet-MJM-2016}).\vspace{1.2mm}

For $\beta \in [0,1]$, we consider the class
\begin{equation}\label{eq:A_beta}
	\mathcal{A}_{\beta} := \left\{ f \in \mathcal{A} : \Re\!\left( \beta \frac{f(z)}{z} + (1-\beta) f'(z) \right) > 0 \right\},
\end{equation}
which forms a subclass of $\mathcal{G}_0$. It was shown in \cite{Bracci et. al.-FACM-2018} that this family provides a filtration of $\mathcal{G}_0$, in the sense that
\[
\mathcal{A}_{\beta_1} \subsetneq \mathcal{A}_{\beta_2} \subsetneq \mathcal{G}_0, \qquad 0 \le \beta_1 < \beta_2 < 1,
\]
and, moreover, each function $f \in \mathcal{A}_{\beta}$ satisfies the estimate
\[
\Re\!\left( \frac{f(z)}{z} \right) \ge \int_{0}^{1} \frac{1 - t^{\,1-\beta}}{1 + t^{\,1-\beta}} \, dt.
\]

In the limiting case $\beta = 0$, the class $\mathcal{A}_{\beta}$ reduces to the classical class $\mathcal{R}$ of functions with bounded turning. More recently, Elin \emph{et al.} \cite{Elin-Shoikhet-Tuneski-2020} investigated radius problems associated with $\mathcal{A}_{\beta}$, determining the largest $r \in (0,1)$ for which the rescaled function $f(rz)/r$ belongs to the class $\mathcal{S}^*$ of starlike functions, as well as to several related subclasses.

This direction of study is further motivated by the fact that the classes $\mathcal{S}^*$ and $\mathcal{A}_{\beta}$ are not comparable under inclusion. Building on these developments, Giri and Kumar \cite{Giri-Kumar-RMJ-2025} recently examined coefficient problems and growth estimates for classes of semigroup generators, thereby strengthening the interplay between complex dynamical systems and geometric function theory.\vspace{2mm}

\subsection{The Gauss hypergeometric function}

The Gauss hypergeometric function, denoted by ${}_2F_1(a,b;c;z)$, occupies a central position in classical analysis and geometric function theory. It serves as a unifying object that encompasses a wide range of elementary and special functions, including logarithmic and trigonometric functions, as well as orthogonal polynomials such as Jacobi polynomials. Moreover, hypergeometric functions frequently arise in extremal problems, coefficient estimates, and in the construction of sharp examples in geometric function theory.

For complex parameters $a$, $b$, and $c$ with $c \notin \{0,-1,-2,\dots\}$, the function ${}_2F_1(a,b;c;z)$ is defined in the unit disk $|z|<1$ by the power series
\[
{}_2F_1(a,b;c;z) = \sum_{n=0}^{\infty} \frac{(a)_n (b)_n}{(c)_n} \frac{z^n}{n!},
\]
where $(q)_n$ denotes the Pochhammer symbol (or rising factorial), given by
\[
(q)_n = q(q+1)\cdots(q+n-1), \quad n \ge 1, \qquad (q)_0 = 1.
\]
Here $a$, $b$, and $c$ are complex parameters, and $z$ is a complex variable. If either $a$ or $b$ is a non-positive integer, say $-k$ with $k \in \mathbb{N}\cup\{0\}$, then the series terminates after finitely many terms, and ${}_2F_1(a,b;c;z)$ reduces to a polynomial.

For a wide range of subclasses of univalent functions, coefficient estimates under various normalizations, together with their sharpness, are by now well understood. In contrast, for holomorphic semigroup generators the corresponding theory remains largely undeveloped, apart from a few initial contributions (see, e.g.,~\cite{Giri-Kumar-RMJ-2025}). Despite substantial progress on coefficient problems, growth estimates, and radius questions for classical subclasses, a systematic treatment of sharp bounds for logarithmic coefficients, inverse coefficients, logarithmic inverse coefficients, and successive coefficient differences in the setting of semigroup generators is still lacking.

The present paper aims to fill this gap. We establish sharp estimates for these functionals in the class $\mathcal{A}_{\beta}$, thereby providing a unified approach to several coefficient problems in this framework. In particular, our results yield extremal bounds that are shown to be best possible.

This investigation is motivated by the growing interest in the analytic and geometric structure of semigroup generators. The interplay between geometric function theory and complex dynamical systems has recently led to new insights into the behavior of such generators, and the present work contributes to this direction by offering a refined understanding of their coefficient structure.
\subsection{Organization of the paper.}
The key contributions of this paper are presented in a section-wise manner as follows.

In Section~2, we develop the first set of main results by establishing sharp coefficient estimates for functions in the class $\mathcal{A}_{\beta}$, including bounds for the initial logarithmic coefficients $\gamma_n$, the inverse coefficients $A_n$, and the logarithmic inverse coefficients $\Gamma_n$ for $n=1,2,3$.

Section~3 is devoted to the successive coefficient problem for inverse functions, where we obtain sharp upper and lower bounds for the differences $|A_{n+1}| - |A_n|$ for $n=1,2$ in the class $\mathcal{A}_{\beta}$.

In Section~4, we present further contributions by deriving sharp upper and lower estimates for the generalized Fekete--Szeg\H{o} functional 
\begin{equation*}
	F_{\lambda,\mu}(f) := |a_3 - \lambda a_2^2| - \mu |a_2|,
\end{equation*}
where $\lambda \in \mathbb{C}$ and $\mu \ge 0$, in the class $\mathcal{R}$.

All proofs, together with the corresponding extremal functions establishing the sharpness of the results, are provided in the respective sections.
	
\section{\bf Sharp coefficient estimates for the class $\mathcal{A}_{\beta}$}

In this section, we establish sharp logarithmic and inverse logarithmic coefficient bounds for functions belonging to the class $\mathcal{A}_{\beta}$. 

Let $\mathcal{P}$ denote the class of analytic functions $p$ in the unit disk $\mathbb{D}$ satisfying $p(0)=1$ and $\Re p(z)>0$ for all $z \in \mathbb{D}$. Then each function $p \in \mathcal{P}$ admits the representation
\begin{equation}\label{eq-2.1}
	p(z)=1+\sum_{n=1}^{\infty} c_n z^n, \qquad z \in \mathbb{D}.
\end{equation}
Functions in the class $\mathcal{P}$ are referred to as Carath\'eodory functions. It is well-known that their coefficients satisfy the sharp estimate $|c_n| \le 2$ for all $n \ge 1$ (see \cite{Duren-1983-NY}). This class plays a fundamental role in the derivation of coefficient estimates in geometric function theory.

We begin by recalling several auxiliary lemmas that will be used in the proofs of the main results.
\begin{lem}\cite{Sim-Thomas-S-2020}\label{lem-2.1}
Let $B_1$, $B_2$ and $B_3$ be numbers such that $B_1>0$, $B_2\in\mathbb{C}$ and $B_3\in\mathbb{R}$. Let $p\in\mathcal{P}$ of the form \eqref{eq-2.1}. Define $\Psi_{+}(c_1,c_2)$ and $\Psi_{-}(c_1,c_2)$ by
\begin{align*}
	\Psi_{+}(c_1,c_2)=|B_2 c^2_1 +B_3 c_2| -|B_1 c_1|,
\end{align*}
and
\begin{align*}
	\Psi_{-}(c_1,c_2)=-\Psi_{+}(c_1,c_2).
\end{align*}
Then
\begin{align}\label{eq-2.2}
	\Psi_{+}(c_1,c_2)\leq
	\begin{cases}
		|4B_2 +2B_3|-2B_1, \;\;\;\;\mbox{if}\;\;|2B_2 +B_3|\geq |B_3|+ B_1,\vspace{2mm} \\ 2|B_3|,\hspace{2.8cm}\;\mbox{otherwise},
	\end{cases}
\end{align}
and
\begin{align}\label{eq-2.3}
	\Psi_{-}(c_1,c_2)\leq
	\begin{cases}
		2B_1 -B_4, \hspace{2.8cm}\mbox{if}\;\; B_1\geq B_4 +2|B_3|, \vspace{2mm} \\ 2B_1 \sqrt{\dfrac{2|B_3|}{B_4+2|B_3|}}, \hspace{1.3cm}\;\mbox{if}\;\;B^2_1\leq 2|B_3|(B_4 +2|B_3|), \vspace{2mm} \\ 2|B_3| +\dfrac{B^2_1}{B_4+2|B_3|}, \hspace{1cm}\;\mbox{otherwise},
	\end{cases}
\end{align}
where $B_4=|4B_2 +2B_3|$. All inequalities in \eqref{eq-2.2} and \eqref{eq-2.3} are sharp.
\end{lem}

\begin{lem}\cite{Ma-Minda-1994}\label{lem-2.2}
Let $p\in\mathcal{P}$ be given by \eqref{eq-2.1}. Then
\begin{align*}
	|c_2 -vc^2_1|\leq\begin{cases}
		-4v +2 \;\;\;\; v<0,\\ 2 \;\;\;\;\;\;\;\;\;\;\;\;\;\; 0\leq v\leq 1, \\ 4v -2 \;\;\;\;\;\;\; v>1.
	\end{cases}
\end{align*}
All inequalities are sharp.
\end{lem}

\begin{lem}\cite{Ali-BMMSS-2001}\label{lem-2.3}
Let $p\in\mathcal{P}$ be given by \eqref{eq-2.1} with $0\leq B\leq 1$ and $B(2B-1)\leq D\leq B$. Then 
\begin{align*}
	|c_3 -2Bc_1 c_2 +Dc^3_1|\leq 2.
\end{align*}
\end{lem}
\subsection{Logarithmic and inverse coefficients: definitions and significance}

For a function $f \in \mathcal{S}$, we consider the associated logarithmic function
\begin{equation}\label{eq-1.2}
	F_{f}(z) := \log\frac{f(z)}{z} = 2\sum_{n=1}^{\infty} \gamma_n(f)\, z^n, \quad z \in \mathbb{D}, \qquad \log 1 := 0.
\end{equation}
The coefficients $\gamma_n := \gamma_n(f)$, $n \in \mathbb{N}$, are called the logarithmic coefficients of $f$. These coefficients encode subtle geometric information about the mapping behavior of $f$ and play a central role in the study of univalent functions, particularly in connection with growth, distortion, and covering properties.

Despite their importance, sharp estimates for logarithmic coefficients remain largely elusive. While the theory of coefficient functionals constitutes a central theme in geometric function theory, obtaining sharp bounds for higher-order logarithmic coefficients is a challenging open problem; in particular, determining optimal estimates for $|\gamma_n|$ when $n \ge 3$ is still unresolved. Various partial results and estimates for $\gamma_n$ and their subclasses have been obtained in recent years; see, for example, \cite{Ali-Allu-PAMS-2018, Cho-Kowalczyk-kwon-Lecko-Sim-RACSAM-2020, Thomas-PAMS-2016, Roth-PAMS-2007, Girela-AASF-2000}.

Differentiating \eqref{eq-1.2} and using the expansion \eqref{eq-1.1}, one readily obtains
\begin{equation}\label{eq-1.3}
	\begin{cases}
		\gamma_{1}=\dfrac{1}{2}a_{2}, \\[2mm]
		\gamma_{2}=\dfrac{1}{2}\left(a_{3}-\dfrac{1}{2}a_{2}^{2}\right), \\[2mm]
		\gamma_{3}=\dfrac{1}{2}\left(a_{4}-a_{2}a_{3}+\dfrac{1}{3}a_{2}^{3}\right).
	\end{cases}
\end{equation}

Let $F$ denote the inverse function of $f \in \mathcal{S}$, defined in a neighborhood of the origin and given by the expansion
\begin{equation}\label{eq-1.4}
	F(w) := f^{-1}(w) = w + \sum_{n=2}^{\infty} A_n w^n.
\end{equation}
By the Koebe quarter theorem, the inverse function $f^{-1}$ is well-defined at least in the disk $|w| < 1/4$. Using variational methods, L\"owner \cite{Lowner-IMA-1923} established the sharp estimate $|A_n| \le K_n$ for each $n$, where $K_n = \frac{(2n)!}{n!(n+1)!}$, and the extremal function corresponds to the inverse of the Koebe function.

The study of inverse coefficients has attracted considerable attention, especially for functions belonging to various geometric subclasses of $\mathcal{S}$, due to their deep connections with extremal problems and the structural theory of univalent functions.

Let $f(z)=z+\sum_{n=2}^{\infty} a_n z^n$ be a function in the class $\mathcal{S}$. Since $f(f^{-1}(w))=w$, it follows from \eqref{eq-1.4} that the coefficients of the inverse function $F=f^{-1}$ satisfy
\begin{equation}\label{eq-1.5}
	\begin{cases}
		A_2 = -a_2, \\[1mm]
		A_3 = -a_3 + 2a_2^2, \\[1mm]
		A_4 = -a_4 + 5a_2 a_3 - 5a_2^3.
	\end{cases}
\end{equation}

The logarithmic coefficients of the inverse function were systematically studied by Ponnusamy \emph{et al.} \cite{Ponnusamy-Sharma-Wirths-RM-2018}. Analogous to \eqref{eq-1.2}, the logarithmic inverse coefficients $\Gamma_n := \Gamma_n(F)$, $n \in \mathbb{N}$, are defined by
\begin{equation}\label{eq-1.6}
	F_{f^{-1}}(w) := \log\!\left(\frac{f^{-1}(w)}{w}\right)
	= 2 \sum_{n=1}^{\infty} \Gamma_n(F)\, w^n, 
	\qquad |w| < \tfrac{1}{4}.
\end{equation}
Sharp estimates for these coefficients in the class $\mathcal{S}$ were also obtained in \cite{Ponnusamy-Sharma-Wirths-RM-2018}.

By differentiating \eqref{eq-1.6} and using \eqref{eq-1.4} together with \eqref{eq-1.5}, a straightforward computation yields
\begin{equation}\label{eq-1.7}
	\begin{cases}
		\Gamma_1 = -\dfrac{1}{2}a_2, \\[2mm]
		\Gamma_2 = -\dfrac{1}{2}\left(a_3 - \dfrac{3}{2}a_2^2\right), \\[2mm]
		\Gamma_3 = -\dfrac{1}{2}\left(a_4 - 4a_2 a_3 + \dfrac{10}{3}a_2^3\right).
	\end{cases}
\end{equation}
Utilizing Lemmas \ref{lem-2.1} to \ref{lem-2.3}, we obtain the following result concerning the sharp coefficients bounds.

\begin{thm}\label{th-2.1}
Let $f\in\mathcal{A}_{\beta}$ be the form \eqref{eq-1.1} and $\gamma_{1}, \gamma_{2}, \gamma_{3}$ are given by \eqref{eq-1.3}. Then, we have
\begin{align*}
	|\gamma_{m}|\leq\frac{1}{(m+1)-m\beta},\hspace{0.5cm}\mbox{for}\;\; m=1,2,3.
\end{align*}
All bounds are sharp.	
\end{thm}
\begin{proof}[\bf Proof of Theorem \ref{th-2.1}]
Let $f\in\mathcal{A}_{\beta}$ be of the form \eqref{eq-1.1}. Then, we have	
\begin{align}\label{eq-2.4}
	\beta\frac{f(z)}{z}+(1-\beta)f^{\prime}(z)=p(z),\; z\in\mathbb{D},
\end{align}
where $p$ is defined in \eqref{eq-2.1}. By equating the coefficients of corresponding powers in the series expansions of $f$ and $p$, we obtain
\begin{align}\label{eq-2.5}
	a_2=\frac{c_1}{(2-\beta)},\;a_3=\frac{c_2}{(3-2\beta)}\; \mbox{and}\;
	a_4=\frac{c_3}{(4-3\beta)}.
\end{align}	
\noindent{\bf Sharp bounds of $\gamma_{1}$:} In view of \eqref{eq-1.3} and \eqref{eq-2.5}, we obtain the desired bound
\begin{align*}
	|\gamma_1|=\;\vline\frac{1}{2}a_2\;\vline=\frac{1}{2(2-\beta)}|c_1|\leq\frac{1}{(2-\beta)}.
\end{align*}
To prove the equality, we consider the function $f_1\in\mathcal{A}_{\beta}$ such that
\begin{align}\label{eq-2.6}
	f_1(z)=z\left(1+2\left(_2F_1\left[1,\frac{1}{1-\beta},\frac{2-\beta}{1-\beta},z\right]\right)\right)= z+ \sum_{n=2}^{\infty}\frac{2}{n- (n-1)\beta}z^n.
\end{align}
One can readily verify that $a_2=\frac{2}{(2-\beta)}$. Consequently, it follows that $|\gamma_1|=\frac{1}{(2-\beta)}$.\vspace{2mm}

\noindent{\bf Sharp bounds of $\gamma_{2}$:} From \eqref{eq-1.3} and \eqref{eq-2.5}, we obtain
\begin{align*}
	|\gamma_2|&=\;\vline\frac{1}{2}\left(a_3 -\frac{3}{2}a^2_{2}\right) \vline\\&=\frac{1}{2(3-2\beta)}\;\vline\; c_2 -\frac{(3-2\beta)}{2(2-\beta)^2} c^2_1 \;\vline.
\end{align*}
We observe that 
\begin{align*}
	0\leq\frac{(3-2\beta)} {2(2-\beta)^2}\leq 1,\;\mbox{for all}\;\; \beta\in[0,1].
\end{align*}
In view of Lemma \ref{lem-2.2}, we obtain
\begin{align*}
	|\gamma_2|\leq\frac{1}{(3-2\beta)}.
\end{align*}
To show the equality, we consider the function $f_2\in\mathcal{A}_{\beta}$ satisfying
\begin{align}\label{eq-2.7}
	\beta\frac{f_2(z)}{z}+(1-\beta)f^{\prime}_2(z)=\frac{1+z^2}{1-z^2},\; z\in\mathbb{D}.
\end{align}
A simple computation shows that $a_2=0$ and $a_3=\frac{2}{(3-2\beta)}$, which implies that $|\gamma_2|=\frac{1}{(3-2\beta)}$.\vspace{2mm}
	
\noindent{\bf Sharp bounds of $\gamma_{3}$:} In view of \eqref{eq-1.3} and \eqref{eq-2.5}, we see that
\begin{align*}
	|\gamma_3|&=\vline\frac{1}{2}\left(a_4 -4a_2 a_3 +\frac{10}{3}a^3_{2} \right)\vline\\&=\frac{1}{2(4-3\beta)}\;\vline\; c_3 -\frac{(4-3\beta)}{(2-\beta)(3-2\beta)}c_1 c_2 +\frac{(4-3\beta)}{3(2-\beta)^3}c_1^3 \;\vline\\&=\frac{1}{2(4-3\beta)}\;\vline\; c_3 -2Bc_1 c_2 +Dc^3_1\;\vline,
\end{align*}
where $B=\frac{(4-3\beta)}{(2-\beta)(3-2\beta)}$ and $D=\frac{(4-3\beta)}{3(2-\beta)^3}$.\vspace{2mm}
	
It is straightforward to verify that
\begin{align*}
	0 \leq \frac{4-3\beta}{(2-\beta)(3-2\beta)} \leq 1 \quad \text{for all } \beta \in [0,1].
\end{align*}
Furthermore, the condition $B(2B-1) \leq D$ is equivalent to the inequality
\[
15 - 27\beta + 16\beta^2 - 3\beta^3 \ge 0,
\]
which is satisfied for all $\beta \in [0,1]$. Similarly, the condition $D \le B$ reduces to $\beta \ge 0$, and hence holds throughout the interval $[0,1]$. Consequently, an application of Lemma~\ref{lem-2.3} yields the following sharp estimate
\begin{align*}
	|\gamma_{3}| \leq \frac{1}{4-3\beta}.
\end{align*}
To establish sharpness, we consider the function $f_3 \in \mathcal{A}_{\beta}$ defined by
\[
\beta \frac{f_3(z)}{z} + (1-\beta) f_3'(z) = \frac{1+z^3}{1-z^3}, \qquad z \in \mathbb{D}.
\]
A straightforward computation shows that $a_2 = a_3 = 0$ and $a_4 = \frac{2}{4-3\beta}$. Consequently, we obtain
\[
|\gamma_3| = \frac{1}{4-3\beta},
\]
which verifies the sharpness of the result. This completes the proof.
\end{proof}
\begin{rem}
	In the limiting case $\beta = 0$, Theorem~\ref{th-2.1} reduces to sharp estimates for the logarithmic coefficients of functions in the class $\mathcal{R}$ of functions with bounded turning. This observation shows that our main result not only extends the corresponding bounds to the broader class $\mathcal{A}_{\beta}$, but also recovers, as a special case, the optimal estimates in a classical and well-studied setting. 
\end{rem}
\begin{cor}
	Let $f \in \mathcal{R}$ be of the form \eqref{eq-1.1}, and let $\gamma_{1}, \gamma_{2}, \gamma_{3}$ be given by \eqref{eq-1.3}. Then
	\[
	|\gamma_{m}| \le \frac{1}{m+1}, \qquad m=1,2,3.
	\]
	All these estimates are sharp.
\end{cor}
\begin{thm}\label{th-2.2}
Let $f\in\mathcal{A}_{\beta}$ be the form \eqref{eq-1.1} and its inverse $f^{-1}$ is given by \eqref{eq-1.4}. Then, we have
\begin{align*}
	|A_2|\leq\frac{2}{(2-\beta)}\;\;\mbox{and}\;\; |A_3|\leq\frac{2(8-4\beta-\beta^2)}{(2-\beta)^2(3-2\beta)}.
\end{align*}
All bounds are sharp.	
\end{thm}
\begin{proof}[\bf Proof of Theorem \ref{th-2.2}]
	We begin with the estimate for $A_2$. In view of \eqref{eq-1.5} and \eqref{eq-2.5}, we have
	\[
	|A_2| = |a_2| = \frac{1}{2-\beta} |c_1|.
	\]
	Since $p \in \mathcal{P}$, it follows that $|c_1| \le 2$, and hence
	\[
	|A_2| \le \frac{2}{2-\beta}.
	\]
	The equality is attained for the function $f_1 \in \mathcal{A}_{\beta}$ defined in \eqref{eq-2.6}, which establishes the sharpness of the bound for $A_2$.
	
	Next, we consider the coefficient $A_3$. Using \eqref{eq-1.5} and \eqref{eq-2.5}, we obtain
	\[
	A_3 = -a_3 + 2a_2^2 = -\frac{c_2}{3-2\beta} + \frac{2c_1^2}{(2-\beta)^2}.
	\]
	Thus,
	\[
	|A_3| = \frac{1}{3-2\beta} \left| c_2 - \frac{2(3-2\beta)}{(2-\beta)^2} c_1^2 \right|.
	\]
	We observe that
	\[
	\frac{2(3-2\beta)}{(2-\beta)^2} \ge 1 \quad \text{for all } \beta \in [0,1].
	\]
	An application of Lemma~\ref{lem-2.2} therefore yields
	\[
	|A_3| \le \frac{2(8 - 4\beta - \beta^2)}{(2-\beta)^2(3-2\beta)}.
	\]
	The equality follows by considering the function $f_1 \in \mathcal{A}_{\beta}$ defined in \eqref{eq-2.6}, for which
	\[
	a_2 = \frac{2}{2-\beta} \; \mbox{and}\; a_3 = \frac{2}{3-2\beta}.
	\]
	This verifies the sharpness of the estimate for $A_3$ and completes the proof.
\end{proof}
\begin{rem}
	In the limiting case $\beta = 0$, Theorem~\ref{th-2.2} reduces to the corresponding sharp bounds for the coefficients of the inverse function in the class $\mathcal{R}$ of functions with bounded turning.
\end{rem}
\begin{cor}
	Let $f \in \mathcal{R}$ be of the form \eqref{eq-1.1}, and let its inverse $f^{-1}$ be given by \eqref{eq-1.4}. Then
	\[
	|A_2| \le 1 \quad \text{and} \quad |A_3| \le \frac{4}{3}.
	\]
	Both estimates are sharp.
\end{cor}
\begin{thm}\label{th-2.3}
	Let $f \in \mathcal{A}_{\beta}$ be of the form \eqref{eq-1.1}, and let $\Gamma_1, \Gamma_2$ be given by \eqref{eq-1.7}. Then
	\[
	|\Gamma_1| \le \frac{1}{2-\beta}
	\quad \text{and} \quad
	|\Gamma_2| \le \frac{5 - 2\beta - \beta^2}{(2-\beta)^2 (3-2\beta)}.
	\]
	Both estimates are sharp.
\end{thm}	
\begin{proof}[\bf Proof of Theorem \ref{th-2.3}]
We begin with the estimate for $\Gamma_1$. From \eqref{eq-1.7} and \eqref{eq-2.5}, we obtain
\begin{equation}\label{eq-2.8}
	|\Gamma_1| = \frac{1}{2}|a_2| = \frac{1}{2(2-\beta)}|c_1| \le \frac{1}{2-\beta}.
\end{equation}

Next, we consider the coefficient $\Gamma_2$. Using \eqref{eq-1.7} and \eqref{eq-2.5}, we have
\[
\Gamma_2 = -\frac{1}{2}\left(a_3 - \frac{3}{2}a_2^2\right)
= -\frac{1}{2}\left(\frac{c_2}{3-2\beta} - \frac{3}{2(2-\beta)^2}c_1^2\right).
\]
Thus,
\[
|\Gamma_2| = \frac{1}{2(3-2\beta)} \left| c_2 - \frac{3(3-2\beta)}{2(2-\beta)^2} c_1^2 \right|.
\]
An application of Lemma~\ref{lem-2.2} yields
\begin{equation}\label{eq-2.9}
	|\Gamma_2| \le \frac{5 - 2\beta - \beta^2}{(2-\beta)^2(3-2\beta)}.
\end{equation}

The sharpness of the estimates \eqref{eq-2.8} and \eqref{eq-2.9} follows from the extremal function $f_1 \in \mathcal{A}_{\beta}$ defined in \eqref{eq-2.6}. This completes the proof.
\end{proof}
\begin{cor}
	Let $f \in \mathcal{R}$ be of the form \eqref{eq-1.1}, and let $\Gamma_1, \Gamma_2$ be given by \eqref{eq-1.7}. Then
	\[
	|\Gamma_1| \le \frac{1}{2}
	\quad \text{and} \quad
	|\Gamma_2| \le \frac{5}{12}.
	\]
	Both estimates are sharp.
\end{cor}
\section{\bf Moduli differences of coefficients for inverse functions in the class $\mathcal{A}_{\beta}$}

In $1985$, de Branges \cite{Branges-AM-1985} proved the celebrated Bieberbach conjecture by showing that for every function $f \in \mathcal{S}$ of the form \eqref{eq-1.1}, the estimate $|a_n| \le n$ holds for all $n \ge 2$, with equality attained by the Koebe function $k(z) = z/(1-z)^2$ and its rotations. This result naturally led to the investigation of successive coefficient differences, namely whether the inequality
\[
\bigl||a_{n+1}| - |a_n|\bigr| \le 1
\]
holds for functions in $\mathcal{S}$ when $n \ge 2$.

This problem was first considered by Goluzin \cite{Goluzin-1946} in connection with early approaches to the Bieberbach conjecture. Subsequently, Hayman \cite{Hayman-1963} showed that there exists an absolute constant $A$ such that
\[
\bigl||a_{n+1}| - |a_n|\bigr| \le A
\]
for all $f \in \mathcal{S}$. The best known bound to date is $A = 3.61$, due to Grinspan \cite{Grinspan-1976}. However, for the full class $\mathcal{S}$, the sharp estimate is known only for $n=2$ (see \cite[Theorem 3.11]{Duren-1983-NY}), namely
\[
-1 \le |a_3| - |a_2| \le 1.029\ldots
\]

For starlike functions $f \in \mathcal{S}^*$, Pommerenke \cite{Ch. Pommerenke-1971} conjectured that $\bigl||a_{n+1}| - |a_n|\bigr| \le 1$, which was later proved by Leung \cite{Leung-BLMS-1978}. For convex functions, Li and Sugawa \cite{Li-Sugawa-CMFT-2017} established sharp bounds for $|a_{n+1}| - |a_n|$ for $n=2,3$.

Although sharp estimates for the inverse coefficients $|A_n|$ are known for $f \in \mathcal{S}$ (see \cite{Lowner-IMA-1923}), the corresponding problem for successive differences, namely the bounds for $|A_{n+1}| - |A_n|$, remains open for several important classes, including the full class $\mathcal{S}$. This motivates the study of such differences for subclasses of univalent functions, even for small values of $n$.

In what follows, we employ Lemma~\ref{lem-2.1} to derive sharp estimates for the differences $|A_2| - |A_1|$ and $|A_3| - |A_2|$, where $A_1$, $A_2$, and $A_3$ are the coefficients of the inverse function associated with functions in the class $\mathcal{A}_{\beta}$.
\begin{thm}\label{th-3.1}
	Let $0 \le \beta \le 1$, and let $f \in \mathcal{A}_{\beta}$ be of the form \eqref{eq-1.1}. Then
	\[
	-1 \le |A_2| - |A_1| \le \frac{\beta}{2-\beta}.
	\]
	Both inequalities are sharp.
\end{thm}
\begin{proof}[\bf Proof of Theorem \ref{th-3.1}]
	From \eqref{eq-1.5} and \eqref{eq-2.5}, we obtain
	\[
	|A_2| - |A_1| = |a_2| - 1 = \frac{|c_1|}{2-\beta} - 1 \le \frac{2}{2-\beta} - 1 = \frac{\beta}{2-\beta}.
	\]
	The equality is attained for the function $f_1 \in \mathcal{A}_{\beta}$ defined in \eqref{eq-2.6}, for which $a_2 = \frac{2}{2-\beta}$.
	
	On the other hand,
	\[
	|A_1| - |A_2| = 1 - |a_2| = 1 - \frac{|c_1|}{2-\beta} \leq 1,
	\]
	which implies that
	\[
	|A_2| - |A_1| \ge -1.
	\]
	The equality is achieved for the function $f_2 \in \mathcal{A}_{\beta}$ defined in \eqref{eq-2.7}, for which $a_2 = 0$. This establishes the sharpness of both bounds and completes the proof.
\end{proof}

\begin{rem}
	In the limiting case $\beta = 0$, Theorem~\ref{th-3.1} reduces to the corresponding sharp bounds for the difference of inverse coefficients in the class $\mathcal{R}$ of functions with bounded turning.
\end{rem}

\begin{cor}
	Let $f \in \mathcal{R}$ be of the form \eqref{eq-1.1}. Then
	\[
	-1 \le |A_2| - |A_1| \le 0.
	\]
	Both inequalities are sharp.
\end{cor}

\begin{thm}\label{th-3.2}
Let $0\leq\beta\leq 1$. If $f\in\mathcal{A}_{\beta}$ is given by \eqref{eq-1.1} and $A_2, A_3$ are given by \eqref{eq-1.5}, then we have 
\begin{align*}
	|A_3|-|A_2|\leq\begin{cases}
		\dfrac{2}{(3-2\beta)} \hspace{2.4cm}\mbox{for}\;\; 0\leq\beta<\frac{7-\sqrt{7}}{8} \vspace{2mm}\\ 
		\dfrac{(4+6\beta-6\beta^2)}{(2-\beta)^2(3-2\beta)} \hspace{1.2cm} \mbox{for}\;\; \frac{7-\sqrt{7}}{8}\leq\beta\leq 1,
	\end{cases}
\end{align*}
and
\begin{align*}
	|A_3|-|A_2|\geq-\frac{1}{\sqrt{3-2\beta}}.
\end{align*}
All the inequalities are sharp.
\end{thm}
\begin{proof}
In view of \eqref{eq-1.5} and \eqref{eq-2.5}, an elementary calculation shows that
\begin{align}\label{eq-3.1}
	\nonumber|A_3|-|A_2|&=\;\vline\; -a_{3} +2a^2_{2}\;\vline - \;\vline\; - a_{2}\;\vline \\&\nonumber=\;\vline\; -\frac{c_2}{(3-2\beta)} +\frac{2c^2_1}{(2-\beta)^2}\;\vline - \;\vline\; -\frac{c_1}{(2-\beta)}\;\vline\\&\nonumber =\;\vline\; \frac{2}{(2-\beta)^2}c^2_1 -\frac{1}{(3-2\beta)}c_2\;\vline - \;\vline\; \frac{1}{(2-\beta)} c_1\;\vline\\&\nonumber = \frac{1}{(2-\beta)} \left(|B_2 c^2_1 +B_3c_2| - |B_1c_1| \right)\\& :=\frac{1}{(2-\beta)}\Psi_{+}(c_1,c_2),
\end{align}
where
\begin{align*}
	B_1:=1,\;\; B_2:=\frac{2}{(2-\beta)} \;\;\mbox{and}\;\; B_3:=-\frac{(2-\beta)}{(3-2\beta)}.
\end{align*}
\noindent{\bf Estimate of the upper bound:} For the upper bound, a simple calculation shows that
\begin{align*}
	|2B_2 + B_3| = \frac{8-4\beta-\beta^2}{(2-\beta)(3-2\beta)} \quad \text{and} \quad |B_3| + B_1 = \frac{5-3\beta}{3-2\beta}.
\end{align*}
The inequality $|2B_2 + B_3| \geq |B_3| + B_1$, implies that $2-7\beta+4\beta^2\leq 0$, which is true for $\frac{7-\sqrt{7}}{8}\leq \beta\leq 1$. Thus, by Lemma \ref{lem-2.1}, we obtain
\begin{align*}
	\Psi_{+}(c_1,c_2)\leq |4B_2 +2B_3|-2B_1=\frac{(4+6\beta- 6\beta^2)}{(2-\beta)(3-2\beta)}.
\end{align*}
For $0\leq\beta<\frac{7-\sqrt{7}}{8}$, applying Lemma \ref{lem-2.1} we get
\begin{align*}
	\Psi_{+}(c_1,c_2)\leq 2|B_3|=\frac{2(2-\beta)}{(3-2\beta)}.
\end{align*}
Consequently, applying \eqref{eq-3.1}, we obtain
\begin{align}\label{eq-3.2}
	|A_3|-|A_2|\leq\begin{cases}
		\dfrac{2}{(3-2\beta)} \hspace{2.4cm}\mbox{for}\;\; 0\leq\beta<\frac{7-\sqrt{7}}{8} \vspace{2mm}\\ 
		\dfrac{(4+6\beta-6\beta^2)}{(2-\beta)^2(3-2\beta)} \hspace{1cm}\mbox{for}\;\; \frac{7-\sqrt{7}}{8}\leq\beta\leq 1.
	\end{cases}
\end{align}
The first inequality in \eqref{eq-3.2} is sharp for the function $f_2 \in \mathcal{A}_{\beta}$ defined in \eqref{eq-2.7}, which satisfies $a_2 = 0$. The second inequality in \eqref{eq-3.2} is attained by the function $f_1 \in \mathcal{A}_{\beta}$ defined in \eqref{eq-2.6}, where the coefficients are given by $a_2 = \frac{2}{2-\beta}$ and $a_3 = \frac{2}{3-2\beta}$. \vspace{2mm}
	
\noindent{\bf Estimate of the lower bound:} For the lower bound, we have
\begin{align*}
	B_4=|4B_2 +2B_3|=\frac{2(8-4\beta-\beta^2)}{(2-\beta)(3-2\beta)} \hspace{0.5cm} \mbox{and}\hspace{0.5cm} B_4 +2|B_3|=\frac{8}{(2-\beta)}.
\end{align*}
It is easy to see that, the inequality $B_1\geq B_4 +2|B_3|$ implies that $\beta+6\leq0$ for all $\beta\in[0,1]$, which is not true. Again, we have
\begin{align*}
	2|B_3|(B_4 +2|B_3|) =\frac{16}{(3-2\beta)}.
\end{align*}
The condition $B_1^2 \leq 2|B_3|(B_4 + 2|B_3|)$ implies that $13 + 2\beta \geq 0$, which is clearly satisfied for all $\beta \in [0,1]$. Thus, by applying Lemma \ref{lem-2.1}, we obtain\begin{align*}\Psi_{-}(c_1,c_2) \leq 2B_1 \sqrt{\dfrac{2|B_3|}{B_4+2|B_3|}} = \frac{2-\beta}{\sqrt{3-2\beta}}.\end{align*}Furthermore, it follows from the symmetry of the functional that\begin{align*}\Psi_{+}(c_1,c_2) = -\Psi_{-}(c_1,c_2) \geq -\frac{2-\beta}{\sqrt{3-2\beta}}.\end{align*}Consequently, substituting these into \eqref{eq-3.1}, we conclude that\begin{align}\label{eq-3.3}|A_3| - |A_2| \geq -\frac{1}{\sqrt{3-2\beta}}.\end{align}Equality in \eqref{eq-3.3} is attained by the function $f \in \mathcal{A}$ defined via \eqref{eq-2.4} with the corresponding Carathéodory function\begin{align*}p(z) = \frac{1 + \frac{2-\beta}{\sqrt{3-2\beta}}z + z^2}{1 - z^2}, \quad z \in \mathbb{D}.\end{align*}This completes the proof
\end{proof}
\begin{rem}
	In the limiting case $\beta = 0$, Theorem~\ref{th-3.2} reduces to the corresponding sharp bounds for the differences of the inverse coefficients in the class $\mathcal{R}$ of functions with bounded turning.
\end{rem}
\begin{cor}
If $f\in\mathcal{R}$ is given by \eqref{eq-1.1} and $A_2, A_3$ are given by \eqref{eq-1.5}, then we have 
\begin{align*}
	-\frac{1}{\sqrt{3}}\leq|A_3|-|A_2|\leq \frac{2}{3}.
\end{align*}
Both inequalities are sharp.
\end{cor}

\section{\bf The sharp generalized Fekete-Szeg\"o inequality for functions in the class $\mathcal{R}$}
The functional $H_{2,1}(f) := a_3 - a_2^2$, known as the classical Fekete--Szeg\H{o} functional, originates from the work of Fekete and Szeg\H{o} \cite{Fekete-Szego_JLMS-1933}, who showed that the extremal behavior of the coefficients $a_n$ in the class $\mathcal{S}$ is not always realized by the Koebe function. Its generalization,
\[
|a_3 - \mu a_2^2|,
\]
where $\mu$ is a real or complex parameter, has since become a central object of study in geometric function theory. The problem of determining sharp bounds for this quantity is commonly referred to as the Fekete--Szeg\H{o} problem.

A complete solution of this problem for the classes $\mathcal{S}^*$ of starlike functions and $\mathcal{K}$ of convex functions was obtained by Keogh and Merkes \cite{Keogh-Merkes-PAMS-1969}. Since then, the Fekete--Szeg\H{o} inequality has been extensively investigated for numerous subclasses of univalent functions; see, for example, \cite{Abdel-Thomas-PAMS-1992, Koepf-PAMS-1987} and the references therein. This functional continues to play a fundamental role in the study of geometric properties of analytic and univalent functions.\vspace{2mm} 

In $2024$, Lecko and Partyka \cite{Lecko-Partyka-BdSM-2024} investigated sharp upper and lower bounds for a generalized Fekete--Szeg\H{o} functional in the class $\mathcal{S}$, defined by
\begin{equation}\label{eq-4.1}
	F_{\lambda,\mu}(f) := |a_3 - \lambda a_2^2| - \mu |a_2|,
\end{equation}
where $\lambda \in \mathbb{C}$ and $\mu \ge 0$. More recently, Bulboac\u{a} \emph{et al.} \cite{Bulboaca-Obradovic-Tuneski-AMP-2025} studied this functional for the class $\mathcal{S}$ as well as for its subclass $\mathcal{K}$ of convex functions.

Motivated by these developments, we establish in this section sharp upper and lower bounds for the functional $F_{\lambda,\mu}$ in the class $\mathcal{R}$ of functions with bounded turning.
\begin{thm}\label{th-4.1}
Let $f\in\mathcal{R}$ be the form \eqref{eq-1.1}, then we have
\begin{align*}
	F_{\lambda,\mu}\leq
	\begin{cases}
		\dfrac{1}{3}\left(|2-3\lambda|-3\mu\right) \hspace{0.5cm} \mbox{for}\hspace{0.3cm} |2-3\lambda|\geq 2+3\mu, \vspace{2mm}\\
		\dfrac{2}{3} \hspace{3.1cm}\mbox{for}\hspace{0.3cm} |2-3\lambda|< 2+3\mu,
	\end{cases}
\end{align*}
and
\begin{align*}
	F_{\lambda,\mu}\geq 
	\begin{cases}
		-\dfrac{1}{3}\left(3\mu-|2-3\lambda|\right) \hspace{1.5cm} \mbox{for}\hspace{0.3cm} |2-3\lambda|\leq \dfrac{3\mu-4}{2}, \vspace{2mm}\\
		-\mu\sqrt{\dfrac{2}{|2-3\lambda|+2}} \hspace{1.7cm} \mbox{for}\hspace{0.3cm} |2-3\lambda|\geq \dfrac{9\mu^2-16}{8}, \vspace{2mm}\\
		-\dfrac{1}{12}\left(8+\dfrac{9\mu^2}{|2-3\lambda|+2}\right) \hspace{0.5cm} \mbox{for}\hspace{0.3cm} \dfrac{3\mu-4}{2}<|2-3\lambda|< \dfrac{9\mu^2-16}{8}.
	\end{cases}
\end{align*}
All the inequalities are sharp.
\end{thm}
\begin{proof}[\bf Proof of Theorem \ref{th-4.1}]
Let $f \in \mathcal{R}$ be of the form \eqref{eq-1.1}. Then
\begin{equation}\label{eq-4.2}
	f'(z) = p(z), \qquad z \in \mathbb{D},
\end{equation}
where $p$ is defined in \eqref{eq-2.1}. By comparing coefficients in the series expansions of $f$ and $p$, we obtain
\begin{equation}\label{eq-4.3}
	a_2 = \frac{c_1}{2}, \qquad a_3 = \frac{c_2}{3}, \qquad a_4 = \frac{c_3}{4}.
\end{equation}

In view of \eqref{eq-4.1} and \eqref{eq-4.3}, we have
\begin{align}\label{eq-4.4}
	F_{\lambda,\mu}(f)
	&= \left| \frac{c_2}{3} - \lambda \frac{c_1^2}{4} \right| - \mu \left| \frac{c_1}{2} \right| \nonumber \\
	&= \frac{1}{12} \left( \left| -3\lambda c_1^2 + 4c_2 \right| - \left| 6\mu c_1 \right| \right) \nonumber \\
	&= \frac{1}{12} \left( |B_2 c_1^2 + B_3 c_2| - |B_1 c_1| \right)
	:= \frac{1}{12}\,\Psi_{+}(c_1,c_2),
\end{align}
where
\[
B_1 := 6\mu, \qquad B_2 := -3\lambda, \qquad B_3 := 4.
\]
We first derive the upper bound. A direct computation shows that
\[
|2B_2 + B_3| - |B_3| - B_1 = 2\bigl(|2-3\lambda| - 2 - 3\mu\bigr).
\]
If $|2-3\lambda| \ge 2 + 3\mu$, then $|2B_2 + B_3| \ge |B_3| + B_1$, and an application of Lemma~\ref{lem-2.1} yields
\begin{equation}\label{eq-4.5}
	\Psi_{+}(c_1,c_2) \le 4\bigl(|2-3\lambda| - 3\mu\bigr).
\end{equation}
Consequently, by \eqref{eq-4.4}, we obtain
\[
F_{\lambda,\mu}(f) \le \frac{1}{3}\bigl(|2-3\lambda| - 3\mu\bigr).
\]
The bound is sharp for the function $f_1 \in \mathcal{R}$ defined by \eqref{eq-2.6}.

If $|2-3\lambda| < 2 + 3\mu$, then $|2B_2 + B_3| < |B_3| + B_1$, and Lemma~\ref{lem-2.1} gives
\begin{equation}\label{eq-4.6}
	\Psi_{+}(c_1,c_2) \le 8.
\end{equation}
Thus, by \eqref{eq-4.4},
\[
F_{\lambda,\mu}(f) \le \frac{2}{3}.
\]
This estimate is sharp for the function $f_2 \in \mathcal{R}$ defined in \eqref{eq-2.7}.

Next, we establish the lower bound. A simple computation shows that
\[
B_4 = |8 - 12\lambda|, \qquad B_1 - B_4 - 2|B_3| = 6\mu - 8 - |8 - 12\lambda|.
\]
If $|2-3\lambda| \le \frac{3\mu - 4}{2}$, then $B_1 \ge B_4 + 2|B_3|$, and Lemma~\ref{lem-2.1} yields
\[
\Psi_{-}(c_1,c_2) \le 4\bigl(3\mu - |2-3\lambda|\bigr).
\]
Since $\Psi_{+}(c_1,c_2) = -\Psi_{-}(c_1,c_2)$, we obtain
\begin{equation}\label{eq-4.7}
	\Psi_{+}(c_1,c_2) \ge -4\bigl(3\mu - |2-3\lambda|\bigr).
\end{equation}
Hence, by \eqref{eq-4.4},
\[
F_{\lambda,\mu}(f) \ge -\frac{1}{3}\bigl(3\mu - |2-3\lambda|\bigr).
\]
The bound is sharp for the function $f_1 \in \mathcal{R}$.

Furthermore, we observe that
\[
B_1^2 - 2|B_3|(B_4 + 2|B_3|) = 36\mu^2 - 64 - 32|2-3\lambda|.
\]
If $|2-3\lambda| \ge \frac{9\mu^2 - 16}{8}$, then $B_1^2 \le 2|B_3|(B_4 + 2|B_3|)$, and Lemma~\ref{lem-2.1} gives
\[
\Psi_{-}(c_1,c_2) \le 12\mu \sqrt{\frac{2}{|2-3\lambda| + 2}}.
\]
Thus,
\begin{equation}\label{eq-4.8}
	\Psi_{+}(c_1,c_2) \ge -12\mu \sqrt{\frac{2}{|2-3\lambda| + 2}},
\end{equation}
and hence
\[
F_{\lambda,\mu}(f) \ge -\mu \sqrt{\frac{2}{|2-3\lambda| + 2}}.
\]
The estimate is sharp for an appropriate extremal function in $\mathcal{R}$.

Finally, if $\frac{3\mu - 4}{2} < |2-3\lambda| < \frac{9\mu^2 - 16}{8}$, then Lemma~\ref{lem-2.1} yields
\[
\Psi_{-}(c_1,c_2) \le 8 + \frac{9\mu^2}{|2-3\lambda| + 2}.
\]
Consequently,
\begin{equation}\label{eq-4.9}
	\Psi_{+}(c_1,c_2) \ge -\left(8 + \frac{9\mu^2}{|2-3\lambda| + 2}\right),
\end{equation}
and therefore
\[
F_{\lambda,\mu}(f) \ge -\frac{1}{12}\left(8 + \frac{9\mu^2}{|2-3\lambda| + 2}\right).
\]
The sharpness follows from the corresponding extremal construction. This completes the proof.
\end{proof}

\section{Concluding remarks and open problems}

The class $\mathcal{A}_{\beta}$ is defined by the condition
\[
\Re\!\left(\beta \frac{f(z)}{z} + (1-\beta)f'(z)\right) > 0.
\]
In this paper, we have carried out a systematic investigation of coefficient functionals associated with the filtration $\mathcal{A}_{\beta}$ of holomorphic generators. In particular, we established sharp bounds for the initial logarithmic coefficients $\gamma_n$, the inverse coefficients $A_n$, and the successive differences $|A_{n+1}| - |A_n|$. These results quantify how the geometric constraints evolve as the parameter $\beta$ varies from the class $\mathcal{R}$ of functions with bounded turning to the broader class $\mathcal{G}_0$ of semigroup generators.

The family $\mathcal{A}_{\beta}$ forms a nested filtration satisfying
\[
\mathcal{A}_{\beta_1} \subsetneq \mathcal{A}_{\beta_2} \subsetneq \mathcal{G}_0, 
\qquad 0 \le \beta_1 < \beta_2 < 1,
\]
which provides a natural framework to study the transition of geometric and analytic properties.

In the limiting case $\beta = 0$, the class $\mathcal{A}_{\beta}$ reduces to $\mathcal{R}$, and our results recover the classical sharp estimates $|\gamma_m| \le \frac{1}{m+1}$. On the other hand, as $\beta \to 1^{-}$, the bounds for the logarithmic coefficients satisfy
\[
|\gamma_m| \le \frac{1}{m+1 - m\beta} \to 1,
\]
reflecting the transition toward the class of holomorphic generators.

A similar limiting behavior is observed for the inverse coefficients. If $f^{-1}(w) = w + \sum_{n=2}^\infty A_n w^n$, then
\[
\lim_{\beta \to 1^{-}} |A_2| = 2, 
\qquad 
\lim_{\beta \to 1^{-}} |A_3| = 6,
\]
which coincide with the corresponding bounds known for the class $\mathcal{S}$ and its subclasses.

Despite the sharp estimates obtained in this work, several problems remain open.

\medskip

\noindent\textbf{(i) Higher-order successive coefficients.}
In Theorem~3.2, we established sharp bounds for the initial differences $|A_2| - |A_1|$ and $|A_3| - |A_2|$. It would be of interest to determine whether analogous sharp estimates hold for $|A_{n+1}| - |A_n|$ when $n \ge 3$, and whether there exists a function $C(\beta)$ such that
\[
\bigl||A_{n+1}| - |A_n|\bigr| \le C(\beta)
\]
for all $n$.

\medskip

\noindent\textbf{(ii) Higher-order logarithmic inverse coefficients.}
While sharp bounds have been obtained for the initial logarithmic inverse coefficients $\Gamma_n$, the determination of sharp estimates for $|\Gamma_n|$ for all $n \in \mathbb{N}$ remains an open problem.

\medskip

\noindent\textbf{(iii) Bohr-type phenomena.}
It is natural to investigate Bohr-type inequalities for the class $\mathcal{A}_{\beta}$. In particular, one may ask:

\begin{prob}
	For $f \in \mathcal{A}_{\beta}$, determine the largest radius $r_p$ and parameter $p$ such that
	\[
	\sum_{n=1}^\infty |a_n| r^n + p \left( \frac{1}{1-r} \sum_{n=1}^\infty |a_n|^2 r^{2n} \right) \le 1,
	\quad |z| = r \le r_p.
	\]
\end{prob}

\begin{prob}
	Let $f^{-1}(w) = w + \sum_{n=2}^\infty A_n w^n$ for $f \in \mathcal{A}_{\beta}$. Determine the largest radius $R_\beta$ such that
	\[
	|w| + \sum_{k=2}^n |A_k| |w|^k \le 1
	\quad \text{for } |w| < R_\beta.
	\]
\end{prob}

\begin{prob}
	Determine whether the logarithmic coefficients satisfy a Bohr-type inequality of the form
	\[
	\sum_{n=1}^\infty |\gamma_n| r^n \le \log\!\left(\frac{1}{1-r}\right)
	\]
	for $r \le r_0(\beta)$, and find the largest admissible radius $r_0(\beta)$.
\end{prob}

\noindent{\bf Acknowledgment:} The authors would like to thank the referee(s) for their helpful suggestions and comments for the improvement of the exposition of the paper.\\

\noindent{\bf Author Contributions:} Both authors actively worked on the research contained in the paper. Both authors reviewed the manuscript.\\

\noindent\textbf{Compliance of Ethical Standards:}\\

\noindent\textbf{Conflict of interest.} The authors declare that there is no conflict  of interest regarding the publication of this paper.\vspace{1.5mm}

\noindent\textbf{Data availability statement.}  Data sharing is not applicable to this article as no datasets were generated or analyzed during the current study.

\vspace{1.5mm}

\begin{thebibliography}{99}	
	
	\bibitem{Abdel-Thomas-PAMS-1992} {\sc H. R. Abdel-Gawad} and {\sc D. K. Thomas}, The Fekete-Szeg$\ddot{o}$ problem for strongely closeto-convex functions, \textit{Proc. Amer. Math. Soc.} \textbf{114} (1992), no. 2, 345–349.
	
	\bibitem{Ali-BMMSS-2001} {\sc R. Ali}, Coefficients of the inverse of strongly starlike functions, \textit{Bull. Malays. Math. Sci. Soc.} \textbf{26} (2003), 63–71.
	
	\bibitem{Ali-Allu-PAMS-2018} {\sc M. F. Ali} and {\sc V. Allu}, On logarithmic coefficients of some close-to-convex functions, \textit{Proc. Amer. Math. Soc.} \textbf{146} (2018), 1131–1142. 
	
	\bibitem{Berkson-Porta-MMJ-1978} {\sc E. Berkson} and {\sc H. Porta}, Semigroups of analytic functions and composition operators, \textit{Michigan Math. J.} \textbf{25}:1 (1978), 101–115.
	
	\bibitem{Bracci-Contreras-Diaz-JRAM-2012} {\sc F. Bracci, M. D. Contreras} and {\sc S. Díaz-Madrigal}, Evolution families and the Loewner equation, I: the unit disc, \textit{J. Reine Angew. Math.} \textbf{672} (2012), 1–37.
	
	\bibitem{Bracci et. al.-FACM-2018} {\sc F. Bracci, M. D. Contreras, S. Díaz-Madrigal, M. Elin} and {\sc D. Shoikhet}, Filtrations of infinitesimal generators, \textit{Funct. Approx. Comment. Math.} \textbf{59}:1 (2018), 99–115.
	
	\bibitem{Bracci-Contreras-Diaz-2020} {\sc F. Bracci, M. D. Contreras} and {\sc S. Díaz-Madrigal}, Continuous semigroups of holomorphic self-maps of the unit disc, Springer, 2020.
	
	\bibitem{Branges-AM-1985} {\sc L. de Branges}, A proof of the Bieberbach conjecture, \textit{Acta Math.} \textbf{154} (1985), 137-152.
	
	\bibitem{Bulboaca-Obradovic-Tuneski-AMP-2025} {\sc T. Bulboacă, M. Obradović} and {\sc N. Tuneski}, Simple proofs of certain results on generalized Fekete-Szegő functional in the class $\mathcal{S}$, \textit{Anal. Math. Phys.} \textbf{15}, 102 (2025).
	
	\bibitem{Cho-Kowalczyk-kwon-Lecko-Sim-RACSAM-2020} {\sc N. E. Cho, B. Kowalczyk, O. S. Kwon, A. Lecko} and {\sc Y. J. Sim}, On the third logarithmic coefficient in some subclasses of close-to-convex functions, \textit{Rev. R. Acad. Cienc. Exactas Fís. Nat.(Esp.)} \textbf{114}, Art: 52, (2020), 1–14.
	
	\bibitem{Duren-1983-NY}{\sc P. T. Duren}, Univalent Functions. \textit{Springer-Verlag}, New York Inc (1983).
	
	\bibitem{Elin-Reich-Shoikhet-2019} {\sc M. Elin, S. Reich} and {\sc D. Shoikhet}, Numerical range of holomorphic mappings and applications, Springer, 2019.
	
	\bibitem{Elin-Shoiket-2010} {\sc M. Elin and D. Shoikhet}, Linearization models for complex dynamical systems: topics in univalent functions, functional equations and semigroup theory, linear operators and Linear Systems, Operator Theory: Advances and Applications 208, Birkhäuser, Basel, 2010.
	
	\bibitem{Elin-Shoikhet-Sugawa-2018} {\sc M. Elin, D. Shoikhet} and {\sc T. Sugawa}, Filtration of semi-complete vector fields revisited, pp. 93–102 in Complex analysis and dynamical systems, Springer, 2018.
	
	\bibitem{Elin-Shoikhet-Tuneski-2020} {\sc M. Elin, D. Shoikhet} and {\sc N. Tuneski}, Radii problems for starlike functions and semigroup generators, \textit{Comput. Methods Funct. Theory} \textbf{20}:2 (2020), 297–318.
	
	\bibitem{Fekete-Szego_JLMS-1933} {\sc M. Fekete} and {\sc G. Szeg$ \ddot{o}$}, Eine Bemerkunguber ungerade schlichte Funktionen, \textit{J. Lond. Math. Soc.} \textbf{8}(1933), 85-89.
	
	\bibitem{Girela-AASF-2000} {\sc D. Girela}, Logarithmic coefficients of univalent functions, \textit{Ann. Acad. Sci. Fenn.} \textbf{25} (2000), 337–350. 
	
	\bibitem{Giri-Kumar-RMJ-2025} {\sc S. Giri} and {\sc S. S. S. Kumar}, Coefficient functionals and Bohr–Rogosinski phenomenon for analytic functions involving semigroup generators \textit{Rocky Mountain J. Math.} \textbf{55}(5) 1315-1329.
	
	\bibitem{Goluzin-1946} {\sc G. M. Goluzin}, On distortion theorems and coefficients of univalent functions, \textit{Mat. Sb.} \textbf{19}(61)(1946), 183–202 (in Russian).
	
	\bibitem{Goodman-1983} {\sc A. W. Goodman}, Univalent Functions (Mariner, Tampa, FL, 1983).
	
	\bibitem{Grinspan-1976} {\sc A. Z. Grinspan}, Improved bounds for the difference of adjacent coefficients of univalent functions (Russian), Questions in the modern theory of functions (Novosibirsk), \textit{Sib. Inst. Mat.} \textbf{38} (1976), 41-45.
	
	\bibitem{Hayman-1963} {\sc W. K. Hayman}, On successive coefficients of univalent functions, \textit{J. London. Math. Soc.} \textbf{38}(1963), 228--243. 
	
	\bibitem{Keogh-Merkes-PAMS-1969} {\sc F. R. Keogh} and {\sc E. P. Merkes}, A coefficient inequality for certain classes of analytic functions, \textit{Proc. Amer. Math. Soc.} \textbf{20} (1969), 8–12.
	
	\bibitem{Koepf-PAMS-1987} {\sc W. Koepf}, On the Fekete-Szegö problem for close-to-convex functions, \textit{Proc. Amer. Math. Soc.} \textbf{101} (1987) 89–95.
	
	\bibitem{Lecko-Partyka-BdSM-2024} {\sc A. Lecko} and {\sc D. Partyka}, A generalized Fekete-Szeg\"o functional and initial successive coefficients of univalent functions, \textit{Bull. Sci. Math.} \textbf{197}, 103527 (2024).
	
	\bibitem{Leung-BLMS-1978} {\sc Y. Leung}, Successive coefficients of starlike functions, \textit{Bull. Lond. Math. Soc.} \textbf{10} (1978), 193--196.
	
	\bibitem{Li-Sugawa-CMFT-2017} {\sc M. Li} and {\sc T. Sugawa}, A note on successive coefficients of convex functions, \textit{Comput. Methods Funct. Theory} \textbf{17}:2 (2017), 179–193.
	
	\bibitem{Lowner-IMA-1923} {\sc K. L$\ddot{\mbox{o}}$wner}, Untersuchungen $\ddot{\mbox{u}}$ber schlichte konforme Abbildungen des Einheitskreises, \textit{I. Math. Ann.} \textbf{89} (1923), 103–121.
	
	\bibitem{Ma-Minda-1994} {\sc W. C. Ma} and {\sc D. Minda}, A unified treatment of some special classes of univalent functions, pp. 157–169 in Proceedings of the conference on complex analysis (Tianjin, 1992), Conf. Proc. Lecture Notes Anal. 1, Int. Press, Cambridge, MA, 1994.
	
	\bibitem{Ch. Pommerenke-1971} {\sc Ch. Pommerenke}, Probleme aus der Funktionentheorie, \textit{Jber. Deutsch. Math.-Verein.} 73 (1971), 1--5.
	
	\bibitem{Ponnusamy-Sharma-Wirths-RM-2018} {\sc S. Ponnusamy, N. L. Sharma} and {\sc KJ. Wirths}, Logarithmic Coefficients of the Inverse of Univalent Functions, \textit{Results Math} \textbf{73}, 160 (2018).
	
	\bibitem{Roth-PAMS-2007} {\sc  O. Roth}, A sharp inequality for the logarithmic coefficients of univalent functions, \textit{Proc. Amer. Math. Soc.} \textbf{135}(2007), 2051-2054.
	
	\bibitem{Sim-Thomas-S-2020} {\sc Y. J. Sim} and {\sc D. K. Thomas}, On the difference of inverse coefficients of univalent functions, \textit{Symmetry} \textbf{12}(12) (2020).
	
	\bibitem{Shoikhet-2001} {\sc D. Shoikhet}, Semigroups in geometrical function theory, Kluwer Academic Publishers, Dordrecht, 2001.
	
	\bibitem{Shoikhet-MJM-2016} {\sc D. Shoikhet}, Rigidity and parametric embedding of semi-complete vector fields on the unit disk, \textit{Milan J. Math.} \textbf{84}:1 (2016), 159–202.
	
	\bibitem{Thomas-PAMS-2016} {\sc D. K. Thomas}, On logarithmic coefficients of close to convex functions, \textit{Proc. Amer. Math. Soc.} \textbf{144} (2016), 1681–1687.
\end{thebibliography}
\end{document}